\documentclass[12pt,twoside]{amsart}
\usepackage{amsmath, amsthm, amscd, amsfonts, amssymb, graphicx}
\usepackage{enumerate}
\usepackage[colorlinks=true,
linkcolor=blue,
urlcolor=cyan,
citecolor=red]{hyperref}
\usepackage{mathrsfs}
\newtheorem{theorem}{Theorem}

\newtheorem{lemma}{Lemma}[section]
\newtheorem{remark}{Remark}[section]

\newtheorem{corollary}{Corollary}[section]
\newtheorem{example}{Example}[section]
\newtheorem{proposition}{Proposition}[section]
\numberwithin{equation}{section}

\begin{document}
\title{A Complementary Inequality to the Information Monotonicity for Tsallis Relative Operator Entropy}
\author{Hamid Reza Moradi$^1$ and Shigeru Furuichi$^2$}
\subjclass[2010]{Primary 47A63, Secondary 47A30, 47A64, 46L05, 47A12.}
\keywords{Tsallis relative operator entropy, L\"owner-Heinz inequality, positive operator, positive linear map, norm inequality.} 
\maketitle

\begin{abstract}
We establish a reverse inequality for Tsallis relative operator entropy involving a positive linear map.  In addition, we present converse of Ando's inequality, for each parameter. We give examples to compare our results with the known results by Furuta and Seo.
In particular, we establish an extension and a reverse of the L\"owner-Heinz inequality under certain condition. Some interesting consequences of inner product spaces and norm inequalities are also presented.
\end{abstract}
\pagestyle{myheadings}
\markboth{\centerline {A Complementary Inequality to the Information Monotonicity ...}}
{\centerline {H.R. Moradi \& S. Furuichi}}
\bigskip
\bigskip
\section{\bf Introduction}
\vskip 0.4 true cm
This paper continues the study of Tsallis relative operator entropy started in \cite{10}.
Let $\mathcal{B}\left( \mathcal{H} \right)$ be the ${{C}^{*}}$-algebra of all (bounded linear) operators on a Hilbert space $\mathcal{H}$. The order relation $A\le B$ for $A,B\in \mathcal{B}\left( \mathcal{H} \right)$ means that both $A$ and $B$ are self-adjoint and $B-A$ is positive. Therefore $0\le A$ means that $A$ is positive. Further, $0<A$ means that $0\le A$ and $A$ is invertible. 

For two invertible positive operators $A$ and $B$ and $p \in \mathbb{R}$ with $p \neq 0$, the Tsallis relative operator entropy ${{T}_{p}}\left( A|B \right)$ is defined by
	\[{{T}_{p}}\left( A|B \right):=\frac{A{{\natural}_{p}}B-A}{p},\qquad \text{ where }\quad A{{\natural}_{p}}B={{A}^{\frac{1}{2}}}{{\left( {{A}^{-\frac{1}{2}}}B{{A}^{-\frac{1}{2}}} \right)}^{p}}{{A}^{\frac{1}{2}}}.\] 
In what follows we use the usual symbol $\sharp_p$ instead of $\natural_p$  for the case $0 \leq p \leq 1$. Also we use $\sharp$ instead of $\sharp_{1/2}$ for simplicity.
Here ${{\mathbf{1}}_{\mathcal{H}}}$ denotes the identity operator on $\mathcal{H}$.

The Tsallis relative operator entropy enjoys the following property for $-1 \leq p \leq 1$ with $p \neq 0$ (see \cite[Proposition 2.3 ]{11} and \cite[Theorem 3.1]{FS2016}):
\begin{equation}\label{20}
\Phi \left( {{T}_{p}}\left( A|B \right) \right)\le {{T}_{p}}\left( \Phi \left( A \right)|\Phi \left( B \right) \right),
\end{equation}
where $\Phi $ is a unital positive linear map on $\mathcal{B}\left( \mathcal{H} \right)$.

Furuta proved in \cite[Theorem 2.1]{9} the following two reverse inequalities involving Tsallis
relative operator entropy ${{T}_{p}}\left( A|B \right)$ via generalized Kantorovich constant $K\left( p \right)$:

\medskip

Let $A$ and $B$ be two positive invertible operators such that $0<{{m}_{1}}{{\mathbf{1}}_{\mathcal{H}}}\le A\le {{M}_{1}}{{\mathbf{1}}_{\mathcal{H}}}$ and $0<{{m}_{2}}{{\mathbf{1}}_{\mathcal{H}}}\le B\le {{M}_{2}}{{\mathbf{1}}_{\mathcal{H}}}$. Put $m=\frac{{{m}_{2}}}{{{M}_{1}}}$, $M=\frac{{{M}_{2}}}{{{m}_{1}}}$, $h=\frac{M}{m}$ and $p\in \left( 0,1 \right]$. Let $\Phi $ be a unital positive linear map on $\mathcal{B}\left( \mathcal{H} \right)$. Then the following inequalities hold:
	\[{{T}_{p}}\left( \Phi \left( A \right)|\Phi \left( B \right) \right)\le \left( \frac{1-K\left( p \right)}{p} \right)\Phi \left( A \right){{\natural}_{p}}\Phi \left( B \right)+\Phi \left( {{T}_{p}}\left( A|B \right) \right),\] 
and
	\[{{T}_{p}}\left( \Phi \left( A \right)|\Phi \left( B \right) \right)\le F\left( p \right)\Phi \left( A \right)+\Phi \left( {{T}_{p}}\left( A|B \right) \right),\] 
where
\[K\left( p \right)=\frac{\left( {{h}^{p}}-h \right)}{\left( p-1 \right)\left( h-1 \right)}{{\left( \frac{\left( p-1 \right)\left( {{h}^{p}}-1 \right)}{p\left( {{h}^{p}}-h \right)} \right)}^{p}},\qquad \text{ }F\left( p \right)=\frac{{{m}^{p}}}{p}\left( \frac{{{h}^{p}}-h}{h-1} \right)\left( 1-K{{\left( p \right)}^{\frac{1}{p-1}}} \right).\]
There are a few other results in this direction; see, e.g., \cite{fu1, 4}. In the present paper, we give alternative bounds for Furuta's inequalities. 

Section \ref{s1} presents a reverse inequality of \eqref{20} when $0<m{{\mathbf{1}}_{\mathcal{H}}}\le {{\mathbf{1}}_{\mathcal{H}}}\le {{A}^{-\frac{1}{2}}}B{{A}^{-\frac{1}{2}}}\le M{{\mathbf{1}}_{\mathcal{H}}}$. We also present two inequalities related to L\"owner-Heinz inequality. Our main idea, and technical tool, is Lemma \ref{l3} below. Most of the results below are rather straightforward consequences of Lemma \ref{l3}. Section \ref{s2}, a related but independent complement, gives several norm and inner product inequalities.

\section{\bf Main Results}\label{s1}
\vskip 0.4 true cm

The following lemma plays a crucial role in our proofs.
\begin{lemma}\label{l3}
Let $0<m\le 1\le t\le M$.
\begin{itemize}
\item[(i)]  If $p\le 1$ with $p \neq 0$, then
\begin{equation}\label{5}
{{M}^{p-1}}\left( t-1 \right)\le \frac{{{t}^{p}}-1}{p}\le {{m}^{p-1}}\left( t-1 \right).
\end{equation}

\item[(ii)] If $p \ge 1$, then 
\begin{equation}\label{6}
{{m}^{p-1}}\left( t-1 \right)\le \frac{{{t}^{p}}-1}{p}\le {{M}^{p-1}}\left( t-1 \right).
\end{equation}
\end{itemize}
\end{lemma}
\begin{proof}
Assume that $f:I\to \mathbb{R}$ is a continuous differentiable function such that $\alpha \le f'\left( t \right)\le \beta $ where $\alpha ,\beta \in \mathbb{R}$. It is an evident fact that two functions ${{g}_{\alpha }}\left( t \right)=f\left( t \right)-\alpha t$ and ${{g}_{\beta }}\left( t \right)=\beta t-f\left( t \right)$ are monotone increasing functions, i.e.,
\begin{equation}\label{1}
a\le b\text{ }\Rightarrow \text{ }\left\{ \begin{aligned}
  & {{g}_{\alpha }}\left( a \right)\le {{g}_{\alpha }}\left( b \right) \\ 
 & {{g}_{\beta }}\left( a \right)\le {{g}_{\beta }}\left( b \right) \\ 
\end{aligned} \right.\text{ }\Leftrightarrow \text{ }\left\{ \begin{aligned}
  & f\left( a \right)-\alpha a\le f\left( b \right)-\alpha b \\ 
 & \beta a-f\left( a \right)\le \beta b-f\left( b \right) \\ 
\end{aligned}, \right.
\end{equation}
for any $a,b\in \left[ m,M \right]\subseteq I$ where $0<m\le M$. 

Letting $f\left( x \right)\equiv {{x}^{p}}$ with $x\in \left[ m,M \right]$ and $0<p\le 1$, a little calculation leads to
\begin{equation}\label{2}
0< m \le a\le b \le M \text{ }\Rightarrow \text{ }\left\{ \begin{aligned}
  & {{a}^{p}}-p{{M}^{p-1}}a\le {{b}^{p}}-p{{M}^{p-1}}b \\ 
 & p{{m}^{p-1}}a-{{a}^{p}}\le p{{m}^{p-1}}b-{{b}^{p}} \\ 
\end{aligned} \right..
\end{equation}
Dividing the both sides in two inequalities given in \eqref{2} by $a^p$ and taking $t=\frac{b}{a}$, we get 
\[0<\frac{m}{a}\le 1\le t\le \frac{M}{a}\text{ }\Rightarrow \text{ }\left\{ \begin{aligned}
  & p{{\left( \frac{M}{a} \right)}^{p-1}}\left( t-1 \right)\le {{t}^{p}}-1 \\ 
 & {{t}^{p}}-1\le p{{\left( \frac{m}{a} \right)}^{p-1}}\left( t-1 \right) \\ 
\end{aligned} \right..\]
Setting $\frac{m}{a}$ and $\frac{M}{a}$ again $m$ and $M$ respectively, we obtain
the desired inequalities in \eqref{5}. 
For the case of $p < 0$, since $f(x) = x^p$ is decreasing and we find $\alpha =pm^{p-1}$ and $\beta =pM^{p-1}$ in the setting of $g_{\alpha}(t)=f(t)-\alpha t$ and $g_{\beta}(t)=\beta t-f(t)$, we get similarly \eqref{2} which implies \eqref{5}.

For the case of $p \ge 1$,  since $f(x) = x^p$ is increasing and we find $\alpha =pm^{p-1}$ and $\beta =pM^{p-1}$ in the setting of $g_{\alpha}(t)=f(t)-\alpha t$ and $g_{\beta}(t)=\beta t-f(t)$, we get similarly
\[0<m\le a\le b\le M\text{ }\Rightarrow \text{ }\left\{ \begin{aligned}
  & {{a}^{p}}-p{{m}^{p-1}}a\le {{b}^{p}}-p{{m}^{p-1}}b \\ 
 & p{{M}^{p-1}}a-{{a}^{p}}\le p{{M}^{p-1}}b-{{b}^{p}} \\ 
\end{aligned} \right..\]
This implies the inequalities \eqref{6}.

\end{proof}

\medskip

From the preceding result, one may derive an interesting operator inequality:
\begin{theorem}\label{thaa}
Let $A,B\in \mathcal{B}\left( \mathcal{H} \right)$ be two positive operators such that $0<m{{\mathbf{1}}_{\mathcal{H}}}\le {{\mathbf{1}}_{\mathcal{H}}}\le {{A}^{-\frac{1}{2}}}B{{A}^{-\frac{1}{2}}}\le M{{\mathbf{1}}_{\mathcal{H}}}$ and $\Phi $ be a unital positive linear map on $\mathcal{B}\left( \mathcal{H} \right)$. 
\begin{itemize}
\item[(1)] If $-1 \le p\le 1$ with $p \ne 0$,
\begin{equation}\label{18}
{{T}_{p}}\left( \Phi \left( A \right)|\Phi \left( B \right) \right)\le \Phi \left( {{T}_{p}}\left( A|B \right) \right)+\left( {{m}^{p-1}}-{{M}^{p-1}} \right)\Phi \left( B-A \right).
\end{equation}
\item[(2)] If $1 \le p \le 2$,
\begin{equation} \label{ineq2_theoremA}
\Phi \left( {{T}_{p}}\left( A|B \right) \right)  \le  {{T}_{p}}\left( \Phi \left( A \right)|\Phi \left( B \right) \right) +\left( {{M}^{p-1}}-{{m}^{p-1}} \right)\Phi \left( B-A \right).
\end{equation}
\end{itemize}
\end{theorem}
\begin{proof}
On account of the first inequality in \eqref{5} we infer that
\begin{equation*}
{{M}^{p-1}}\Phi \left( B-A \right)\le \Phi \left( {{T}_{p}}\left( A|B \right) \right),
\end{equation*}
and the second one gives
\begin{equation*}
{{T}_{p}}\left( \Phi \left( A \right)|\Phi \left( B \right) \right)\le {{m}^{p-1}}\Phi \left( B-A \right).
\end{equation*}
Combining above two inequalities with the previous inequality \eqref{20}  
for  $-1 \le p\le 1$ with $p \ne 0$, we have the desired inequality \eqref{18}. In the  case of $p \ge 1$, by \eqref{6} we infer that
\[{{M}^{p-1}}\Phi \left( B-A \right)\ge \Phi \left( {{T}_{p}}\left( A|B \right) \right),\qquad \text{}\quad {{T}_{p}}\left( \Phi \left( A \right)|\Phi \left( B \right) \right)\ge {{m}^{p-1}}\Phi \left( B-A \right).\]
We obtain the inequality \eqref{ineq2_theoremA}, since we have the following relation
\[\Phi \left( {{T}_{p}}\left( A|B \right) \right)\ge {{T}_{p}}\left( \Phi \left( A \right)|\Phi \left( B \right) \right),\qquad \text{ }1\le p\le 2\]
which can be shown by the similar way to the proof of \cite[Theorem 2.2]{FS2016}.
\end{proof}

\medskip

Using the same strategy as in the proof of \cite[Corollary 2]{4}, we get the following converse of Ando's inequality $\Phi \left( A\sharp B \right)\le \Phi \left( A \right)\sharp \Phi \left( B \right)$ (see \cite{3}), for each parameter.
\begin{corollary}\label{thb}
Let the assumptions of Theorem \ref{thaa}  be satisfied. Then we have the following inequalities.
\begin{itemize}
\item[(1)] If $0 \le p \le 1$,
\begin{equation*}
\Phi \left( A \right){{\sharp}_{p}}\Phi \left( B \right)\le \Phi \left( A{{\sharp}_{p}}B \right)+p\left( {{ m }^{p-1}}-{{ M }^{p-1}} \right)\Phi \left( B-A \right).
\end{equation*}
\item[(2)] If $-1 \le p \le 0$, 
$$
\Phi \left( A \right){{\natural}_{p}}\Phi \left( B \right)\ge \Phi \left( A{{\natural}_{p}}B \right)+p\left( {{ m }^{p-1}}-{{ M }^{p-1}} \right)\Phi \left( B-A \right).
$$
\item[(3)] If $1 \le p \le 2$,
\[\Phi \left( A{{\natural}_{p}}B \right) \le \Phi \left( A \right){{\natural}_{p}}\Phi \left( B \right) +p\left( {{M}^{p-1}}-{{m}^{p-1}} \right)\Phi \left( B-A \right).\]
\end{itemize}
\end{corollary}
\begin{remark}
Let $\Phi(X) = \frac{1}{\dim \mathcal{H}} Tr[X]$ and let $A$, $B$ be density operators (which are positive operators with unit trace). Then Corollary \ref{thb} gives
$Tr[A \sharp_p B] \geq 1$ for $0 \le p \le 1$, and $Tr[A \natural_p B] \leq 1$ for $-1 \le p \le 0$ or $1 \leq p \leq 2$.
\end{remark}

\medskip

The following two examples illustrate Theorem \ref{thaa} and Corollary \ref{thb} are nontrivial.
\begin{example}
We compare our result with Furuta's results. We consider $2 \times 2$ matrices.
Take $\Phi \left( X \right)=\frac{1}{2}Tr\left[ X \right]$, $p=\frac{1}{2}$ and
\[A = \left( {\begin{array}{*{20}{c}}
2&3\\
3&5
\end{array}} \right),\qquad B = \left( {\begin{array}{*{20}{c}}
3&4\\
4&6
\end{array}} \right).\]
Then the eigenvalues of $A$ and $B$ are respectively $\frac{7\pm 3\sqrt{5}}{2}$ and $\frac{9\pm\sqrt{73}}{2}$ so that we put $m_1=\frac{7- 3\sqrt{5}}{2}$, $M_1=\frac{7+ 3\sqrt{5}}{2}$, $m_2=\frac{9-\sqrt{73}}{2}$ and $M_2=\frac{9+\sqrt{73}}{2}$.
According to the setting in Introduction, we set $m=\frac{m_2}{M_1} \simeq 0.0332645$, $M=\frac{M_2}{m_1} \simeq 60.1242$ and $h=\frac{M}{m} \simeq 1807.46$.
Since the eigenvalues of ${{A}^{-\frac{1}{2}}}B{{A}^{-\frac{1}{2}}}$ are approximately $1$ and $2$ so that we have the condition $m{{\mathbf{1}}_{\mathcal{H}}} \leq A^{-1/2}BA^{-1/2} \leq M{{\mathbf{1}}_{\mathcal{H}}}$. Then we calculate the following quantities:
\[\frac{1}{2}\left( {{m}^{p-1}}-{{M}^{p-1}} \right)Tr\left[ B-A \right]\simeq 5.35393,\] 
	\[\frac{1}{2}\left( \frac{1-K\left( p \right)}{p} \right)Tr{{\left[ A \right]}^{1-p}}Tr{{\left[ B \right]}^{p}}\simeq 5.55857,\] 
	\[\frac{1}{2}F\left( p \right)Tr\left[ A \right]\simeq 12.6413.\] 
This example shows our result is better than Furuta's ones (at least in this case).
\end{example}
\begin{example}
We also compare our result with Seo's result \cite[Theorem 1]{8}:
	\[\Phi \left( A \right){{\sharp}_{p}}\Phi \left( B \right)-\Phi \left( A{{\sharp}_{p}}B \right)\le -C\left( m,M,p \right)\Phi \left( A \right),\] 
where
\[C\left( {m,M,p} \right) \equiv \left( {p - 1} \right){\left( {\frac{{{M^p} - {m^p}}}{{p\left( {M - m} \right)}}} \right)^{\frac{p}{{p - 1}}}} + \frac{{M{m^p} - m{M^p}}}{{M - m}},\]
for $0 \le p \le 1$ and $mA \le B \le MA$ for some scalar $0 < m \le M$.
 Many examples show Seo's result is better than ours. However, we can find the example such that our result is better than Seo's result in the following. We consider $2 \times 2$ matrices. Setting $\Phi \left( X \right)=\frac{1}{2}Tr\left[ X \right]$, $p=\frac{1}{2}$  and \[A = \left( {\begin{array}{*{20}{c}}
2&3\\
3&5
\end{array}} \right), \qquad B = \left( {\begin{array}{*{20}{c}}
{2.01}&3\\
3&{5.01}
\end{array}} \right).\]
Since the eigenvalues of ${{A}^{-\frac{1}{2}}}B{{A}^{-\frac{1}{2}}}$ are approximately $1.06854$ and $1.001459$, we take $M=1.07$ and $m=0.999$. A little calculation shows that
\[ - C\left( {m,M,p} \right)\Phi \left( A \right) \simeq 0.000524.\]
Also we have
\[p\left( {{m^{p - 1}} - {M^{p - 1}}} \right)\Phi \left( {B - A} \right) \simeq 0.0001688.\]
\end{example}

\medskip

As an immediate consequence of Corollary \ref{thb},  we have the following result.
\begin{corollary}
Replacing $A$ by ${{\mathbf{1}}_{\mathcal{H}}}$ and $B$ by $A$ in Corollary \ref{thb}, we get the following inequalities.
\begin{itemize}
\item[(1)] If $0 \le p \le 1$, then
$\Phi {{\left( A \right)}^{p}}\le \Phi \left( {{A}^{p}} \right)+p\left( {{m}^{p-1}}-{{M}^{p-1}} \right)\left( \Phi \left( A \right)-{{\mathbf{1}}_{\mathcal{H}}} \right).$
\item[(2)] If $-1\le p \le 0$, then 
$\Phi {{\left( A \right)}^{p}}\ge \Phi \left( {{A}^{p}} \right)+p\left( {{m}^{p-1}}-{{M}^{p-1}} \right)\left( \Phi \left( A \right)-{{\mathbf{1}}_{\mathcal{H}}} \right).$
\item[(3)] If $1 \le p \le 2$, then 
$\Phi(A^p) \le \Phi(A)^p +p\left( {{M}^{p-1}}-{{m}^{p-1}} \right)\left( \Phi \left( A \right)-{{\mathbf{1}}_{\mathcal{H}}} \right).$
\end{itemize}
\end{corollary}

\medskip

We recall the following famous inequality \cite{lowner,Hei}:

\medskip

\noindent {\bf Theorem 2.1.} {\it(L\"owner-Heinz inequality) Let $A,B\in \mathcal{B}\left( \mathcal{H} \right)$ be two positive operators. Then 
	\[A\le B\text{ }\Rightarrow \text{ }{{A}^{p}}\le {{B}^{p}},\qquad \text{ }0\le p\le 1.\] 
}	
It is essential to notice that L\"owner-Heinz inequality does not always hold for $p >1$.

\medskip

Using Lemma \ref{l3} we get a kind of extension and reverse of the L\"owner-Heinz inequality under the assumption $\left\| A \right\|{{\mathbf{1}}_{\mathcal{H}}}\le B$ (here $\left\| A \right\|$ stands for the usual operator norm of $A$). For the sake of convenience, we cite a useful lemma which we will use in the below.
\begin{lemma}\label{prop1}
Let $A$ be a positive operator.
\begin{itemize}
\item[(i)] If $0 \leq p \leq 1$, then
\begin{equation}\label{11}
{{A}^{p}}\le {{\left\| A \right\|}^{p}}{{\mathbf{1}}_{\mathcal{H}}}-p{{\left\| A \right\|}^{p-1}}\left( \left\| A \right\|{{\mathbf{1}}_{\mathcal{H}}}-A \right).
\end{equation}
\item[(ii)] If $p \geq 1$ or $p \leq 0$, then
\begin{equation}\label{13}
{{\left\| A \right\|}^{p}}{{\mathbf{1}}_{\mathcal{H}}}-p{{\left\| A \right\|}^{p-1}}\left( \left\| A \right\|{{\mathbf{1}}_{\mathcal{H}}}-A \right)\le {{A}^{p}}.
\end{equation}
\end{itemize} 
\end{lemma}
\begin{proof}
Since $A\le \left\| A \right\|{{\mathbf{1}}_{\mathcal{H}}}$, $M=\left\| A \right\|$ and $A$ commute with ${{\mathbf{1}}_{\mathcal{H}}}$, we can use directly the scalar inequality
$a^p -p M^{p-1} a \leq b^p -p M^{p-1} b$ for $a \leq b$ 
as $a =A$ and $b=\left\| A \right\|{{\mathbf{1}}_{\mathcal{H}}}$. Then we can obtain \eqref{11}. \eqref{13} can be proven by the similar way to use the inequality
$p m^{p-1} a-a^p \leq p m^{p-1} b-b^p$ for $a \leq b$.
\end{proof}

\medskip

Now we come to the announced theorem.
\begin{theorem}\label{thc}
Let $A,B\in \mathcal{B}\left( \mathcal{H} \right)$ be two positive operators such that $\left\| A \right\|{{\mathbf{1}}_{\mathcal{H}}}\le B$. 
\begin{itemize}
\item[(1)] If $0\le p\le 1$, then
\begin{equation}\label{15}
p{{\left\| B \right\|}^{p-1}}\left( B-A \right)\le {{B}^{p}}-{{A}^{p}}.
\end{equation}
\item[(2)] If $p\ge 1$ or $p\le 0$, then
\begin{equation}\label{16}
{{B}^{p}}-{{A}^{p}}\le p{{\left\| B \right\|}^{p-1}}\left( B-A \right).
\end{equation}
\end{itemize}
\end{theorem}
\begin{proof}
Replacing $a$ by $\left\| A \right\|$ and then applying functional calculus for the operator $B$ in the first inequality in \eqref{2}, we get
\[\left\| A \right\|{{\mathbf{1}}_{\mathcal{H}}}\le B\le M{{\mathbf{1}}_{\mathcal{H}}}\text{ }\Rightarrow \text{ }\left( {{\left\| A \right\|}^{p}}-p{{M}^{p-1}}\left\| A \right\| \right){{\mathbf{1}}_{\mathcal{H}}}\le {{B}^{p}}-p{{M}^{p-1}}B.\]
On account of the inequality \eqref{11}, we have
\[{{A}^{p}}-p{{M}^{p-1}}A\le \left( {{\left\| A \right\|}^{p}}-p{{M}^{p-1}}\left\| A \right\| \right){{\mathbf{1}}_{\mathcal{H}}}.\]
This is the same as saying
\[\left\| A \right\|{{\mathbf{1}}_{\mathcal{H}}}\le B\le M{{\mathbf{1}}_{\mathcal{H}}}\text{ }\Rightarrow \text{ }{{A}^{p}}-p{{M}^{p-1}}A\le {{B}^{p}}-p{{M}^{p-1}}B.\]
The choice $M=\left\| B \right\|$ yields \eqref{15}. By the same method the inequality \eqref{16} is obvious by \eqref{13}.

\end{proof}

\medskip

Now, we illustrate Theorem \ref{thc} by the following example. 
\begin{example}
Taking $A=\left( \begin{matrix}
   3 & -1 & 0  \\
   -1 & 2 & 1  \\
   0 & 1 & 1  \\
\end{matrix} \right)$, $B=\left( \begin{matrix}
   9 & 0 & 1  \\
   0 & 6 & 2  \\
   1 & 2 & 7  \\
\end{matrix} \right)$. Then, after a straight forwards computation,
\begin{itemize}
\item[(i)] for $p=\frac{2}{3}$
	\[{{B}^{p}}-{{A}^{p}}-\left( p{{\left\| B \right\|}^{p-1}}\left( B-A \right) \right)\simeq \left( \begin{matrix}
   0.398 & 0.186 & -0.035  \\
   0.186 & 0.541 & -0.225  \\
   -0.035 & -0.225 & 0.842  \\
\end{matrix} \right)\gneqq0.\] 
\item[(ii)] for $p=4$  
	\[p{{\left\| B \right\|}^{p-1}}\left( B-A \right)-\left( {{B}^{p}}-{{A}^{p}} \right)\simeq \left( \begin{matrix}
   14675.664 & 2845.944 & 1333.944  \\
   2845.944 & 12145.776 & 1141.944  \\
   1333.944 & 1141.944 & 17699.664  \\
\end{matrix} \right)\gneqq0.\] 
\item[(iii)] for $p=-3$
\[p{{\left\| B \right\|}^{p-1}}\left( B-A \right)-\left( {{B}^{p}}-{{A}^{p}} \right)\simeq \left( \begin{matrix}
   2.371 & 6.373 & -8.624  \\
   6.373 & 17.366 & -23.62  \\
   -8.624 & -23.62 & 32.367  \\
\end{matrix} \right)\gneqq0.\]
\end{itemize}  
\end{example}

\begin{remark}
Theorem \ref{thc} shows that if $\left\| A \right\| {{\mathbf{1}}_{\mathcal{H}}} \le B$, then
\[0\le \frac{p{{\left\| B \right\|}^{p-1}}}{\left\| {{\left( B-A \right)}^{-1}} \right\|}{{\mathbf{1}}_{\mathcal{H}}}\le p{{\left\| B \right\|}^{p-1}}\left( B-A \right)\le {{B}^{p}}-{{A}^{p}}, \qquad\text{ }0\le p\le 1.\]
We compare this with the following result given in \cite[Corollary 2.5 (i)]{MN2012}
\[0\le {{\left\| B \right\|}^{p}}{{\mathbf{1}}_{\mathcal{H}}}-{{\left( \left\| B \right\|-\frac{1}{\left\| {{\left( B-A \right)}^{-1}} \right\|} \right)}^{p}}{{\mathbf{1}}_{\mathcal{H}}}\le {{B}^{p}}-{{A}^{p}},\]
for $0 \le p \le 1$ and $0\le A<B$.
Since $1\le \left\| B{{\left( B-A \right)}^{-1}} \right\|\le \left\| B \right\|\left\| {{\left( B-A \right)}^{-1}} \right\|\equiv s$, we show
\begin{equation} \label{comp_MN2012_eq01}
p{{s}^{p-1}}\le {{s}^{p}}-{{\left( s-1 \right)}^{p}},\qquad\text{ }0\le p\le 1,\text{ }s\ge 1.
\end{equation}
Putting $t \equiv \frac{1}{s}$, the inequality (\ref{comp_MN2012_eq01}) is equivalent to the inequality
\[{{\left( 1-t \right)}^{p}}\le 1-pt,\qquad\text{ }0\le p\le 1,\text{ }0<t\le 1.\]
This inequality can be proven by putting $f_p(t) \equiv 1-p t -(1-t)^p$ and then calculate $f_p'(t)=p\left\{(1-t)^{p-1}-1\right\} \geq 0$ which implies $f_p(t) \geq f_p(0) =0$. 
The inequality (\ref{comp_MN2012_eq01}) thus implies the relation
\[\frac{p{{\left\| B \right\|}^{p-1}}}{\left\| {{\left( B-A \right)}^{-1}} \right\|}{{\mathbf{1}}_{\mathcal{H}}}\le {{\left\| B \right\|}^{p}}{{\mathbf{1}}_{\mathcal{H}}}-{{\left( \left\| B \right\|-\frac{1}{\left\| {{\left( B-A \right)}^{-1}} \right\|} \right)}^{p}}{{\mathbf{1}}_{\mathcal{H}}}.\]
\end{remark}

\section{\bf More Applications of Lemma \ref{l3}}\label{s2}
\vskip 0.4 true cm

In this section we present many hidden consequences of Lemma \ref{l3}, several of them improving classical inequalities.
\subsection{\bf Some Inequalities in Inner Product Space}
The following is an extension of the result by Mond and Pe\v cari\'c \cite[Theorem 1.2]{5} for convex functions to differentiable functions.
\begin{theorem}
(A weakened version of Mond-Pe\v cari\'c inequality) Let $A,B\in \mathcal{B}\left( \mathcal{H} \right)$  such that $0<m{{\mathbf{1}}_{\mathcal{H}}}\le B\le A\le M{{\mathbf{1}}_{\mathcal{H}}}$. If $f$ is a continuous differentiable function such that $\alpha \le f'\le \beta $ with $\alpha ,\beta \in \mathbb{R}$, then  
\[\alpha \left\langle \left( A-B \right)x,x \right\rangle \le \left\langle f\left( A \right)x,x \right\rangle -f\left( \left\langle Bx,x \right\rangle  \right)\le \beta \left\langle \left( A-B \right)x,x \right\rangle ,\]
for every unit vector $x\in \mathcal{H}$. 
\end{theorem}
\begin{proof}
We follow a similar path as in the proof of Theorem 3.3 in \cite{12}. Due to relation \eqref{1}, we have
	\[a{{\mathbf{1}}_{\mathcal{H}}}\le A\text{ }\Rightarrow \text{ }\left\{ \begin{aligned}
  & f\left( a \right){{\mathbf{1}}_{\mathcal{H}}}-\alpha a{{\mathbf{1}}_{\mathcal{H}}}\le f\left( A \right)-\alpha A \\ 
 & \beta a{{\mathbf{1}}_{\mathcal{H}}}-f\left( a \right){{\mathbf{1}}_{\mathcal{H}}}\le \beta A-f\left( A \right) \\ 
\end{aligned} \right..\] 
So for all $x\in \mathcal{H}$ with $\left\| x \right\|=1$, we have
\begin{equation*}
a\le \left\langle Ax,x \right\rangle \text{ }\Rightarrow \text{ }\left\{ \begin{aligned}
& f\left( a \right)-\alpha a\le \left\langle f\left( A \right)x,x \right\rangle -\alpha \left\langle Ax,x \right\rangle  \\ 
& \beta a-f\left( a \right)\le \beta \left\langle Ax,x \right\rangle -\left\langle f\left( A \right)x,x \right\rangle  \\ 
\end{aligned} \right..
\end{equation*}
With the substitution $a=\left\langle Bx,x \right\rangle $ this becomes
	\[\left\langle Bx,x \right\rangle \le \left\langle Ax,x \right\rangle \text{ }\Rightarrow \text{ }\left\{ \begin{aligned}
  & f\left( \left\langle Bx,x \right\rangle  \right)-\alpha \left\langle Bx,x \right\rangle \le \left\langle f\left( A \right)x,x \right\rangle -\alpha \left\langle Ax,x \right\rangle  \\ 
 & \beta \left\langle Bx,x \right\rangle -f\left( \left\langle Bx,x \right\rangle  \right)\le \beta \left\langle Ax,x \right\rangle -\left\langle f\left( A \right)x,x \right\rangle  \\ 
\end{aligned}, \right.\] 
which is the desired conclusion.
\end{proof}

\medskip

The next theorem will play a role, which was given in \cite{6} in a more general setting.

\medskip

\noindent {\it{\bf Theorem 3.1.} (H\"older-McCarthy inequality) Let $A\in \mathcal{B}\left( \mathcal{H} \right)$ be a positive operator and $x\in \mathcal{H}$ be a unit vector. 
\begin{itemize}
\item[(1)] $\left\langle {{A}^{p}}x,x \right\rangle \le {{\left\langle Ax,x \right\rangle }^{p}}$ for all $0<p<1$.
\item[(2)] ${{\left\langle Ax,x \right\rangle }^{p}}\le \left\langle {{A}^{p}}x,x \right\rangle $ for all $p>1$.
\item[(3)] If $A$ is invertible, then ${{\left\langle Ax,x \right\rangle }^{p}}\le \left\langle {{A}^{p}}x,x \right\rangle $ for all $p<0$.
\end{itemize}
}

\medskip

Using Lemma \ref{l3}, we are able to point out the following reverse and improvement of the  H\"older-McCarthy inequality.
\begin{proposition}
Let $A$ be a positive operator  with $0<m{{\mathbf{1}}_{\mathcal{H}}}\le A\le M{{\mathbf{1}}_{\mathcal{H}}}$ and let $x\in \mathcal{H}$ be a unit vector.
\begin{itemize}
\item[(1)] If $0<p<1$, then
\[\frac{p}{{{M}^{1-p}}}\left( \left\langle Ax,x \right\rangle -{{\left\langle {{A}^{p}}x,x \right\rangle }^{\frac{1}{p}}} \right)\le {{\left\langle Ax,x \right\rangle }^{p}}-\left\langle {{A}^{p}}x,x \right\rangle \le \frac{p}{{{m}^{1-p}}}\left( \left\langle Ax,x \right\rangle -{{\left\langle {{A}^{p}}x,x \right\rangle }^{\frac{1}{p}}} \right).\]

\item[(2)] If $p>1$ or $p<0$, then
\[\frac{p}{{{m}^{1-p}}}\left( {{\left\langle {{A}^{p}}x,x \right\rangle }^{\frac{1}{p}}}-\left\langle Ax,x \right\rangle  \right)\le \left\langle {{A}^{p}}x,x \right\rangle -{{\left\langle Ax,x \right\rangle }^{p}}\le \frac{p}{{{M}^{1-p}}}\left( {{\left\langle {{A}^{p}}x,x \right\rangle }^{\frac{1}{p}}}-\left\langle Ax,x \right\rangle  \right).\]
\end{itemize}
\end{proposition}
\begin{proof}
The first one follows from \eqref{2} by taking $a={{\left\langle {{A}^{p}}x,x \right\rangle }^{\frac{1}{p}}}$, $b=\left\langle Ax,x \right\rangle $ with $0<p<1$. The second one is completely similar as that before, so we leave out the details.
\end{proof}
\subsection{\bf  Norm Inequalities}
Let $A\in \mathcal{B}\left( \mathcal{H} \right)$.  As is well known, 
\begin{equation}\label{14}
\left\| A \right\|\le {{\left\| A \right\|}_{2}}\le {{\left\| A \right\|}_{1}},
\end{equation}
where $\left\| \cdot \right\|$, ${{\left\| \cdot \right\|}_{2}}$, ${{\left\| \cdot \right\|}_{1}}$ are usual operator norm, Hilbert-Schmidt norm and trace class norm, respectively. 

As a consequence of the inequality \eqref{1} we have a refinement and a reverse of \eqref{14} as follows:
\[\left\| A \right\|\le \frac{\left\| A \right\|_{2}^{p}-{{\left\| A \right\|}^{p}}}{p\left\| A \right\|_{1}^{p-1}}+\left\| A \right\|\le {{\left\| A \right\|}_{2}}\le \frac{\left\| A \right\|_{1}^{p}-\left\| A \right\|_{2}^{p}}{p\left\| A \right\|_{1}^{p-1}}+{{\left\| A \right\|}_{2}}\le {{\left\| A \right\|}_{1}},\qquad \text{ }p\ge 1\text{ or }p\le 0\]
and	
\[\frac{\left\| A \right\|_{2}^{p}-\left\| A \right\|_{1}^{p}}{p\left\| A \right\|_{1}^{p-1}}+{{\left\| A \right\|}_{1}}\le {{\left\| A \right\|}_{2}}\le \frac{\left\| A \right\|_{2}^{p}-{{\left\| A \right\|}^{p}}}{p\left\| A \right\|_{1}^{p-1}}+\left\| A \right\|,\qquad \text{ }0<p\le 1.\]

\medskip

\noindent More in the same vein as above, if $A\in \mathcal{B}\left( \mathcal{H} \right)$, then
\[r\left( A \right)\le \frac{w{{\left( A \right)}^{p}}-r{{\left( A \right)}^{p}}}{p{{\left\| A \right\|}^{p-1}}}+r\left( A \right)\le w\left( A \right)\le \frac{{{\left\| A \right\|}^{p}}-w{{\left( A \right)}^{p}}}{p{{\left\| A \right\|}^{p-1}}}+w\left( A \right)\le \left\| A \right\|,\qquad\text{ }p\ge 1\text{ or }p\le 0\]
and
	\[\frac{w{{\left( A \right)}^{p}}-{{\left\| A \right\|}^{p}}}{p{{\left\| A \right\|}^{p-1}}}+\left\| A \right\|\le w\left( A \right)\le \frac{w{{\left( A \right)}^{p}}-r{{\left( A \right)}^{p}}}{p{{\left\| A \right\|}^{p-1}}}+r\left( A \right),\qquad \text{ }0<p\le 1\]
where $r\left( A \right)$, $w\left( A \right)$ and $\left\| A \right\|$ are the spectral radius, numerical radius and the usual operator norm of $A$, respectively. The inequalities above follow from the fact that for any $A\in \mathcal{B}\left( \mathcal{H} \right)$,
\[r\left( A \right)\le w\left( A \right)\le \left\| A \right\|.\]

\medskip

The following norm inequalities are well known.

\medskip

\noindent {\it{\bf Theorem 3.2.}
Let $A,B$ be two positive operators.
\begin{itemize}
\item[(1)] \cite[Theorem IX.2.1]{1} If $0\le p\le 1$, then
\begin{equation}\label{12}
\left\| {{A}^{p}}{{B}^{p}} \right\|\le {{\left\| AB \right\|}^{p}}.
\end{equation}
\item[(2)] \cite[Theorem IX.2.3]{1} If $p\ge 1$, then
\begin{equation}\label{9}
{{\left\| AB \right\|}^{p}}\le \left\| {{A}^{p}}{{B}^{p}} \right\|.
\end{equation}
\end{itemize}
}

\medskip

The following proposition provides a refinement and a reverse for the inequalities \eqref{12} and \eqref{9}. The proof is the same as one of Proposition \ref{prop1} and we omit it.

\begin{proposition}\label{tha}
Let $A,B$ be two positive operators such that $0<m{{\mathbf{1}}_{\mathcal{H}}}\le A,B\le M{{\mathbf{1}}_{\mathcal{H}}}$. Then  
\begin{equation}\label{3}
\frac{p}{{{M}^{1-p}}}\left( \left\| AB \right\|-{{\left\| {{A}^{p}}{{B}^{p}} \right\|}^{\frac{1}{p}}} \right)\le {{\left\| AB \right\|}^{p}}-\left\| {{A}^{p}}{{B}^{p}} \right\|\le \frac{p}{{{m}^{1-p}}}\left( \left\| AB \right\|-{{\left\| {{A}^{p}}{{B}^{p}} \right\|}^{\frac{1}{p}}} \right),
\end{equation}
for any $0\le p\le 1$. Moreover, if $p\ge 1$, then
\begin{equation}\label{10}
\frac{p}{{{m}^{1-p}}}\left( {{\left\| {{A}^{p}}{{B}^{p}} \right\|}^{\frac{1}{p}}}-\left\| AB \right\| \right)\le \left\| {{A}^{p}}{{B}^{p}} \right\|-{{\left\| AB \right\|}^{p}}\le \frac{p}{{{M}^{1-p}}}\left( {{\left\| {{A}^{p}}{{B}^{p}} \right\|}^{\frac{1}{p}}}-\left\| AB \right\| \right).
\end{equation}
 \end{proposition}
 \begin{remark}
 The interested reader can construct other norm (trace and determinant) inequalities using our approach given in Lemma \ref{l3}. We leave the details of this idea to the
 interested reader, as it is just an application of our main results.
 \end{remark}

\noindent{\bf Acknowledgement.} The authors would like to thank the anonymous referee for
valuable comments that have been implemented in the final version of the paper. The author (S.F.) was partially supported by JSPS KAKENHI Grant Number 16K05257.
\bibliographystyle{alpha}

\vskip 0.4 true cm

\tiny {\uppercase{$^1$Young Researchers and Elite Club, Mashhad Branch, Islamic Azad University, Mashhad, Iran.}

{\it E-mail address:} hrmoradi@mshdiau.ac.ir

\vskip 0.4 true cm

\uppercase{$^2$Department of Information Science, College of Humanities and Sciences, Nihon University, 3-25-40, Sakurajyousui, Setagaya-ku, Tokyo, 156-8550, Japan.}

{\it E-mail address:} furuichi@chs.nihon-u.ac.jp}

\end{document}